\documentclass[11pt]{amsart}
\usepackage[top=30truemm,bottom=30truemm,left=30truemm,right=30truemm]{geometry} 
\usepackage{amsmath, amssymb, array}
\usepackage[all]{xy}
\usepackage{ascmac}
\usepackage[dvips]{graphicx}
\usepackage{color}
\usepackage{amscd}
\usepackage[driverfallback=dvipdfm]{hyperref}
\usepackage{amsrefs}
\usepackage{cleveref}
\usepackage{shuffle}
\usepackage{mathrsfs}
\usepackage{listings} 
\usepackage{comment}

\newfont{\cyr}{wncyr10 scaled\magstep0}

\newcommand{\bb}[1]{\boldsymbol{#1}}

\newcommand{\C}{\mathbb{C}}

\newcommand{\F}{\mathbb{F}}

\newcommand{\PP}{\mathbb{P}}
\newcommand{\Q}{\mathbb{Q}}

\newcommand{\Z}{\mathbb{Z}}

\DeclareMathOperator{\Res}{Res} 
\DeclareMathOperator{\Tr}{Tr} 

\theoremstyle{plain}
 \newtheorem{theorem}{Theorem}[section]
 \crefname{theorem}{Theorem}{Theorems}
 \newtheorem{proposition}[theorem]{Proposition}
 \crefname{proposition}{Proposition}{Propositions}
 \newtheorem{lemma}[theorem]{Lemma}
 \crefname{lemma}{Lemma}{Lemmas}
 \newtheorem{corollary}[theorem]{Corollary}
 \crefname{corollary}{Corollary}{Corollaries}
 
 \crefname{conjecture}{Conjecture}{Conjectures}
 
 \crefname{hypothesis}{Hypothesis}{Hypotheses}
 
 \crefname{question}{Question}{Questions}
 
 \crefname{problem}{Problem}{Problems}

\theoremstyle{definition} 
 
 \crefname{definition}{Definition}{Definitions}
 \newtheorem{example}[theorem]{Example}
 \crefname{example}{Example}{Examples}
 \newtheorem{remark}[theorem]{Remark}
 \crefname{remark}{Remark}{Remarks}

\title{Quadratic fields, Artin-Schreier extensions, and Bell numbers}
\author{Yoshinosuke Hirakawa}
\address[Yoshinosuke Hirakawa]{Department of Mathematics \\ Faculty of Science and Technology \\ Tokyo University of Science, 2641, Yamazaki, Noda, Chiba, Japan}
\email{hirakawa\_yoshinosuke@rs.tus.ac.jp \\ hirakawa@keio.jp}
\thanks{This research was supported by the Research Grant of Keio Leading-edge Laboratory of Science \& Technology (Grant Numbers 2018-2019 000036 and 2019-2020 000074).
     This research was supported in part by KAKENHI 18H05233 and 21K13779.
     This research was conducted as part of the KiPAS program FY2014--2018 of the Faculty of Science and Technology at Keio University as well as the JSPS Core-to-Core program ``Foundation of a Global Research Cooperative Center in Mathematics focused on Number Theory and Geometry".}

\begin{document}

\begin{abstract}
In this article,
we prove a modulo $p$ congruence which connects
the class number of the quadratic field $\Q(\sqrt{(-1)^{(p-1)/2}p})$
and the trace of a certain monomial in a root $\theta$ of the Artin-Schreier polynomial $\theta^{p}-\theta-1$ over the field $\F_{p}$ of $p$ elements.
This formula has a flavor of Dirichlet's class number formula
which connects the class number and the $L$-value.
The proof of our formula is based on several formulae satisfied by the Bell number,
where the latter is defined as the number of partitions of $\{ 1, 2, ..., n \}$
and a purely combinatorial object.
Among such formulae,
we prove a generalization of the so called ``trace formula'' due to Barsky and Benzaghou
which describes the special values of the Bell polynomials modulo $p$ by the trace mentioned above.
\end{abstract}

\maketitle

\section{Introduction}

Since the 19th century,
the study of the class numbers of number fields
has been one of the major subjects in number theory.
They are important in themselves,
have many applications to Diophantine problems like Fermat's Last Theorem,
and are related with many other mathematical objects
(see e.g. \cite{BKOR,Chowla,Cox,Feustel-Holzapfel,Girstmair_popular,Hirzebruch-Zagier,Mordell_factorial,Rosen-Shnidman,Shioda-Inose,Shioda-Mitani}).
As far as the author knows, however,
no direct relation is known between
the class numbers of the quadratic extensions of the field $\Q$ of rational numbers
and the Artin-Schreier extension $\F_{p^{p}}$ of
the field $\F_{p}$ of $p$ elements.
In this article, we prove such a formula.

Let $p > 3$ be a prime
and set $p^{*} := (-1)^{(p-1)/2}p$.
Let $( \cdot /p)_{2}$ be the quadratic residue symbol modulo $p$,
$\Tr : \F_{p^{p}} \to \F_{p}$ be the trace map,
and take $\theta \in \F_{p^{p}}$ such that $\theta^{p} = \theta+1$.
Let $h(p^{*})$ be the class number of $\Q(\sqrt{p^{*}})$,
and $\epsilon_{p} = t_{p}+u_{p}\sqrt{p} > 1$ be the fundamental unit of $\Q(\sqrt{p}) \subset \mathbb{R}$
with $t_{p}, u_{p} \in 2^{-1}\Z$.
Finally, set $\tau_{p}(a) := \sum_{j = 1}^{p-1} jp^{ja-1}$.
Then, our formula is stated as follows:

\begin{theorem} \label{main_1}
The following congruence holds for every $a \in \Z$ such that $p \nmid a$;
\[
	\left( \frac{-2a}{p} \right)_{2}\Tr(\theta^{\tau_{p}(a)})
	\equiv \left( \frac{p-1}{2} \right)!
	\equiv \begin{cases}
		(-1)^{\frac{h(p)+1}{2}} t_{p} \bmod p & \text{if $p \equiv 1 \bmod 4$}, \\
		(-1)^{\frac{h(-p)+1}{2}} \bmod p & \text{if $p \equiv -1 \bmod 4$}.
	\end{cases}
\]
\end{theorem}

Since $\tau_{p}(a)+1 \equiv (1-p^{-a})^{-1} \bmod (p^{p}-1)/(p-1)$ (cf.\ \cref{inverse})
is associated with the Riemann zeta function $\zeta(s) := \prod_{p} (1-p^{-s})^{-1}$,
it is not so strange to compare \cref{main_1} with the following consequence of the analytic class number formula (cf.\ \cite[\S49 and \S51]{Hecke}):
\[
	e^{\frac{1}{2}\sqrt{p^{*}} \Res_{s = 1}\zeta_{\Q(\sqrt{p^{*}})}(s)}
	= \begin{cases}
	\epsilon_{p}^{h(p)} & \text{if $p \equiv 1 \bmod 4$}, \\
	\bb i^{h(-p)} & \text{if $p \equiv -1 \bmod 4$ (and $p > 3$)},
	\end{cases}
\]
where $\zeta_{\Q(\sqrt{p^{*}})}(s)$ is the Dedekind zeta function of $\Q(\sqrt{p^{*}})$
and $\bb i$ is the imaginary unit.

In addition, what is interesting is that
our proof of \cref{main_1} is based on some properties of the Bell number $b_{n}$, 
which is defined as the number of partitions of $\{ 1, 2, ..., n \}$
and a purely combinatorial object.
The Bell number $b_{n}$ has a natural subdivision by 
the Stirling numbers $S(n, j)$ of the second kind,
which are the numbers $S(n, j)$ of the partitions of $\{ 1, 2, ..., n \}$ into non-empty $j$ subsets for non-negative integers $j$.
Since $S(n, j) = 0$ for $j > n$, the series
\[
	b_{n}(x)
	:= \sum_{j \geq 0} S(n, j) x^{j},
\]
defines a polynomial in $x$ satisfying $b_{n}(1) = b_{n}$,
which we call the Bell polynomial.
For known properties of these classical sequences,
we refer the reader to \cite{Aigner,Barsky_Bell,Barsky-Benzaghou,Bell_numbers,Bell_polynomials,Carlitz_1955,Comtet,De Angelis-Marcello,Ehrenborg,Gallardo-Rahavandrainy,Kahale,Junod_trace,Mezo_book,Radoux_survey,Radoux_Hankel}
and references therein.

A key ingredient of our proof of \cref{main_1} is the following congruence,
which is a generalization of
the ``trace formula'' due to Barsky and Benzaghou \cite[Th\'eor\`eme 2]{Barsky-Benzaghou}.

\begin{theorem} [\cref{trace_formula}] \label{main_2}
Let $a, m \in \Z_{\geq 0}$ such that $p \nmid a$.
Then, it holds that
\[
	a^{m}b_{m}(a^{-1})
	\equiv - \Tr(\theta^{\tau_{p}(a)})
		\Tr(\theta^{m-1-\tau_{p}(a)}) \bmod p.
\]
\end{theorem}

Here, we should emphasize that it is classically known that $b_{m}(x)$ satisfies a linear recursion
\[
	b_{m+p}(x) \equiv b_{m+1}(x) + x^{p}b_{m}(x) \bmod p
\]
(see \cref{Touchard_congruence}),
and hence it is not surprising that the sequence $(b_{m}(a^{-1}) \bmod p)_{m}$ is described by some traces of polynomials in $\theta$.
\cref{main_2} states, however, that this sequence can be described by the trace of the monomial $\theta^{m-1-\tau_{p}(a)}$ in $\theta$ (up to a subtle arithmetic constant in \cref{main_1}),
which is what is surprising.
\footnote{
	In this view,
	it is more natural to call this formula the ``monomial trace formula.''
	}
Moreover, \cref{main_2} (see also \cref{trace_calculus}) tells us some consecutive zeros as follows:
\begin{equation} \label{consecutuve_zeros}
	b_{\tau_{p}(a)+1}(a^{-1})
	\equiv \cdots
	\equiv b_{\tau_{p}(a)+p-1}(a^{-1})
	\equiv b_{\tau_{p}(a)+p+1}(a^{-1})
	\equiv \cdots
	\equiv b_{\tau_{p}(a)+2p-1}(a^{-1})
	\equiv 0 \bmod p.
\end{equation}
In fact, \cref{consecutuve_zeros}, which was established by Radoux \cite{Radoux_Artin_Schreier} for $a = 1$
and stated by Junod \cite[p.\ 78]{Junod_trace} in general,
is a key ingredient of the proof of \cref{main_1} (see \S4).

\begin{remark}
In fact, the statement of \cref{main_2} itself is essentially equivalent to
Junod's ``theorem'' \cite[Theorem 5]{Junod_trace}.
However, his proof contains several errors including an imprecise intermediate result \cite[Theorem 3]{Junod_trace}.
In this article, we correct these errors.
\end{remark}

In \S2,
we introduce the weighted Bell polynomial (cf.\  \cite{Carlitz_I,Carlitz_II})
as a generalization of the Bell polynomial
and recall its basic properties.
In \S3,
we prove \cref{main_2} in a generalized form for weighted Bell polynomials.
The idea of its proof is the same as that in \cite{Junod_trace}.
However, we shall clarify the crucial ideas.
In \S4,
we first reduce \cref{main_1} to the Hankel determinant formula in \cref{Bell_factorial}.
After that, we prove \cref{Bell_factorial} along Radoux's proof \cite{Radoux_Hankel} for $a = 1$.

\section{Weighted Bell polynomials and the generating series}

Among several generalizations of $S(n, j)$, 
Carlitz \cite{Carlitz_I, Carlitz_II} introduced
the weighted Stirling ``number'' $R(n, j; \lambda)$ of the second kind,
which is a polynomial in $\lambda$ defined by
\[
	R(n, j; \lambda)
	:= \sum_{m \geq 0} \binom{n}{m} S(m, j) \lambda^{n-m}
		\in \Z[\lambda].
\]
Here, $\binom{n}{m}$ is the binomial coefficient.
The classical Stirling number $S(n, j)$ is recovered as the constant term $R(n, j; 0)$.
Since $R(n, j; \lambda)$ is characterized by the initial condition
\[
	\sum_{n \geq 0} R(n, 0; \lambda) z^{n}
	= \sum_{n \geq 0} \lambda^{n}z^{n}
		= \frac{1}{1-\lambda z}
\]
and the recursive condition
\begin{align*}
	R(n, j; \lambda)
	= (\lambda+j) R(n-1, j, \lambda) + R(n-1, j-1, \lambda),
\end{align*}
its ordinary generating series is given by
\begin{equation} \label{GF_Stirling}
	F_{j}(\lambda, z) := \sum_{n \geq 0} R(n, j; \lambda) z^{n}
	= \frac{1}{1-\lambda z} \prod_{i = 1}^{j} \frac{z}{1-(\lambda+i)z}.
\end{equation}
We can identify $R(n, j; \lambda)$ with the coefficients of the asymptotic expansion of the Tate twist of the absolute zeta function $\zeta_{\PP^{j}/\F_{1}}(z-\lambda)$ of the $j$-dimensional projective space \cite{Koyama-Kurokawa_Tate,Kurokawa_F1} because
\[
	z^{-1}F_{j}(\lambda, z^{-1})
	= \prod_{i = 0}^{j} \frac{1}{(z-\lambda)-i}
	= \zeta_{\PP^{j}/\F_{1}}(z-\lambda).
\]

Let $b_{n}(x, \lambda)$ be a polynomial in $x$ and $\lambda$ defined by
\[
	b_{n}(x, \lambda)
	:= \sum_{j \geq 0} R(n, j; \lambda) x^{j},
\]
which is denoted by $r_{n}(t)$ in \cite{Howard}.
We call $b_{n}(x, \lambda)$ the {\it weighted Bell polynomial}.
For related sequences, see e.g.\ \cite[\S7]{Carlitz_I} and \cite{Broder,Mezo,Mezo-Ramirez}.
Since $R(n, j; 0) = S(n, j)$,
we see that $b_{n}(x, 0) = b_{n}(x)$.
Moreover, since
\begin{equation} \label{convolution}
	b_{n}(x, \lambda)
	= \sum_{j \geq 0} \sum_{m \geq 0}
		\binom{n}{m} S(m, j) \lambda^{n-m} x^{j}
			= \sum_{m \geq 0} \binom{n}{m} b_{m}(x) \lambda^{n-m},
\end{equation}
the exponential generating series of $b_{n}(x, \lambda)$ is given by
\[
	E(x, \lambda, z)
	:= \sum_{n \geq 0} b_{n}(x, \lambda) \frac{z^{n}}{n!}
	= e^{\lambda z} \sum_{n \geq 0} b_{n}(x) \frac{z^{n}}{n!}
	= e^{x(e^{z}-1)+\lambda z} \in \Q[\lambda][[x, z]]
\]
(cf.\  \cite[Theorem 3.1]{Mezo}).
On the other hand,
\cref{GF_Stirling} leads us to the following representation of the ordinary generating series of $b_{n}(x, \lambda)$.

\begin{theorem} \label{GF_Bell}
Let $F(x, \lambda, z) := \sum_{n \geq 0} b_{n}(x, \lambda)z^{n}$.
Then, we have
\[
	F(x, \lambda, z)
	= \frac{1}{1-\lambda z} \sum_{n \geq 0} \prod_{j = 1}^{n} \frac{xz}{1-(\lambda+j)z}.
\]
\end{theorem}

Although its proof is obvious,
\cref{GF_Bell} leads us to several significant applications.
For example, we can generalize \cite[Theorem 3.2]{Mezo} in a simpler manner.

\begin{corollary} \label{GF_Bell_Mezo}
The following identity holds in $\Z[\lambda][[x, z]]$:
\[
	F(x, \lambda, z)
	= e^{-x} \sum_{n \geq 0} \frac{x^{n}}{(1-(\lambda+n)z)n!}.
\]
\end{corollary}

\begin{proof}
It is sufficient to prove that
\[
	(1-\lambda z)F(x, \lambda, z)
	= e^{-x} \sum_{n \geq 0} \left( \frac{1-\lambda z}{1-(\lambda +n)z} \cdot \frac{x^{n}}{n!} \right).
\]
Moreover, it is sufficient to compare the coefficients of $x^{n}/n!$ for $n \geq 1$ as elements of $\Z[\lambda][[z]]$,
which boils down to checking the following partial fractional decomposition
\[
	\prod_{j = 1}^{n} \frac{jz}{1-(\lambda+j)z}
	= \sum_{j \geq 1} (-1)^{n-j} \binom{n}{j} \frac{jz}{1-(\lambda+j)z}
		= \sum_{j \geq 0} (-1)^{n-j} \binom{n}{j} \frac{1-\lambda z}{1-(\lambda+j)z}.
\]
Here, the second equality is a consequence of $ \sum_{j \geq 0} (-1)^{n-j} \binom{n}{j} = 0$.
On the other hand, the first equality holds
because if we divide the both sides by $z$,
then we obtain the same rational functions of $z$
\[
	\sum_{j = 1}^{n} \frac{\rho_{j}(\lambda)}{z-(\lambda+j)^{-1}}
\]
which is characterized by simple poles at $z = (\lambda+j)^{-1}$ ($1 \leq j \leq n$)
with residues
\[
	\rho_{j}(\lambda)
	:= - \frac{\frac{j}{\lambda+j}}{\lambda+j} \prod_{\substack{1 \leq i \leq n \\ i \neq j}} \frac{\frac{i}{\lambda+j}}{1-\frac{\lambda+i}{\lambda+j}}
		= - \frac{n!}{(\lambda+j)^{2}} \prod_{\substack{1 \leq i \leq n \\ i \neq j}} \frac{1}{j-i}
			= (-1)^{n-j-1} \binom{n}{j}\frac{j}{(\lambda+j)^{2}}
\]
and a zero at $z = \infty$. This completes the proof.
\end{proof}

\begin{remark}
In \cite[Theorem 3.2]{Mezo},
Mez\"o stated \cref{GF_Bell_Mezo} for $\lambda = r \in \Z_{\geq 0}$
in terms of the hypergeometric function $_{1}F_{1}$
as follows:
\[
	F(x, r, z)
	= \frac{e^{-x}}{1-rz} {}_{1}F_{1}\left( \left. \begin{matrix} r-z^{-1} \\ r+1-z^{-1} \end{matrix} \right| x \right)
		= \frac{e^{-x}}{1-rz} \sum_{n \geq 0} \left( \prod_{j = 0}^{n-1} \frac{r-z^{-1}+j}{r-z^{-1}+j+1} \right) \frac{x^{n}}{n!}.
\]
\end{remark}

An advantage of $F(x, \lambda, z)$ over $E(x, \lambda, z)$ is that
its reduction modulo an arbitrary integer is well-defined,
and it has several applications to arithmetic properties of $b_{n}(x, \lambda)$ modulo integers.
For instance, we can generalize \cite[Theorem 1]{Barsky-Benzaghou} and \cite[Theorem 1]{Junod_trace}
following the same method as the original proofs in \cite{Barsky-Benzaghou,Junod_trace}.

\begin{corollary} \label{rationality}
Let $m \in \Z_{\geq 1}$.
Then, the following congruence holds in $(\Z/m\Z)[x, \lambda][[z]]$:
\begin{equation*} \label{GF_congruence_lambda}
	F(x, \lambda, z)
	\equiv \frac{\sum_{k = 0}^{m-1} (xz)^{k} \prod_{j = k+1}^{m-1} (1-(\lambda+j)z)}{\prod_{j = 0}^{m-1} (1-(\lambda+j)z) - (xz)^{m}}
		\bmod m\Z.
\end{equation*}
In particular, for every prime $p$, every positive integer $n$, and every integer $r$,
the following congruence holds in $(\Z_{p}/(np/2)\Z_{p})[x][[z]]$:
\begin{equation} \label{GF_congruence_r}
	F(x, r, z)
	\equiv \frac{\sum_{k = 0}^{np-1} (xz)^{np-1-k} \prod_{j = 1}^{k} (1-(r-j)z)}{(1-z^{p-1})^{n} - (xz)^{np}}
			\bmod \frac{np}{2}\Z_{p}.
\end{equation}
\end{corollary}

\begin{proof}
First, note that
by substituting $qm+k$ for $n$ in \cref{GF_Bell},
we can decompose the single summation over $n \geq 0$ to a double summation as follows:
\[
	F(x, \lambda, z)
	= \frac{1}{1-\lambda z}\sum_{q \geq 0} \sum_{k = 0}^{m-1}
		\prod_{j = 1}^{qm} \frac{xz}{1-(\lambda+j)z}
			\prod_{j = 1}^{k} \frac{xz}{1-(\lambda+qm+j)z}.
\]
Therefore, we obtain \cref{GF_congruence_lambda} as follows:
\begin{align*}			
	F(x, \lambda, z)
	&\equiv \frac{1}{1-\lambda z} \sum_{k = 0}^{m-1} \prod_{j = 1}^{k} \frac{xz}{1-(\lambda+j)z}
		\sum_{q \geq 0} \left( \prod_{j = 0}^{m-1} \frac{xz}{1-(\lambda+j)z} \right)^{q}
			\bmod m\\
	&\equiv \sum_{k = 0}^{m-1} (xz)^{k} \prod_{j = 0}^{k} (1-(\lambda+j)z)^{-1}
		\left( 1-\prod_{j = 0}^{m-1} \frac{xz}{1-(\lambda+j)z} \right)^{-1} \bmod m \\
	&\equiv \frac{\sum_{k = 0}^{m-1} (xz)^{k} \prod_{j = k+1}^{m-1} (1-(\lambda+j)z)}{\prod_{j = 0}^{m-1} (1-(\lambda+j)z) - (xz)^{m}} \bmod m.
\end{align*}
For \cref{GF_congruence_r},
let $n = n_{0}p^{v}$ with $\nu \in \Z$ and $n_{0} \in \Z \setminus p\Z$.
Then, \cite[Lemma 1.3]{Gertsch-Robert} implies that
\[
	\prod_{j = 0}^{np-1} (1-(r+j)z)
	\equiv \left( \prod_{j = 0}^{p^{v+1}-1} (1-jz) \right)^{n_{0}}
		\equiv ((1-z^{p-1})^{p^{v}})^{n_{0}} \bmod \frac{1}{2}p^{v+1}\Z_{p}.
\]
On the other hand, the following congruence holds:
\begin{align*}
	\sum_{k = 0}^{np-1} (xz)^{k} \prod_{j = k+1}^{np-1} (1-(r+j)z)
	&= \sum_{k = 0}^{np-1} (xz)^{k} \prod_{j = 1}^{np-1-k} (1-(r+np-j)z) \\
	&\equiv \sum_{k = 0}^{np-1} (xz)^{np-1-k} \prod_{j = 1}^{k} (1-(r-j)z) \bmod np,
\end{align*}
By combining them with \cref{GF_congruence_lambda},
we obtain \cref{GF_congruence_r}.
\end{proof}

\begin{example} \label{Touchard_congruence}
Suppose that $p$ is odd, $n = 1$, and $r = 0$.
Then, \cref{GF_congruence_r} implies that
\[
	\left( 1-z^{p-1}-(xz)^{p} \right) \sum_{n \geq 0} b_{n}(x) z^{n}
	\equiv \sum_{k = 0}^{p-1} (xz)^{k} \prod_{j = k+1}^{p-1} (1-jz) \bmod p.
\]
By comparing the coefficients of $z^{n+p}$,
we obtain the congruence \cite[(5)]{Radoux_Touchard} mentioned in the previous section:
\begin{equation} \label{Touchard_1}
	b_{n+p}(x) - b_{n+1}(x) - x^{p}b_{n}(x) \equiv 0 \bmod p.
\end{equation}
In fact,
in a similar manner (with some auxiliary congruences),
we can prove that
\begin{equation} \label{Touchard_2}
	b_{n+p}(x, \lambda) - b_{n+1}(x, \lambda) -  (\lambda^{p}-\lambda+x^{p})b_{n}(x, \lambda)
	\equiv 0 \bmod p.
\end{equation}
In the case of $\lambda \in \Z$,
the above congruence has already appeared in the proof of \cite[Theorem 3.1]{Howard}
and was rediscovered in \cite[Theorem 4]{Mezo-Ramirez}.
In fact,
both \cref{Touchard_1} and \cref{Touchard_2} hold also for $p = 2$,
which we can check more directly (cf.\ \cite{Touchard}, \cite[Lemma 4]{Williams}).
\end{example}

\section{Trace formula for weighted Bell polynomials}

Let $p$ be an arbitrary prime.
In this section, we prove the following ``trace formula''
for weighted Bell polynomials $b_{n}(x, \lambda)$ specialized at $x = a$ for an integer $a$,
which is a generalization of \cite[Th\'eor\`eme 2]{Barsky-Benzaghou}.

\begin{theorem} \label{trace_formula}
Let $a, m \in \Z_{\geq 0}$ such that $p \nmid a$
and $\Tr : \F_{p^{p}}[\lambda] \to \F_{p}[\lambda]$ be the $\F_{p}$-linear extension of $\Tr : \F_{p^{p}} \to \F_{p}$.
Then, the following congruence holds in $\F_{p}[\lambda]$:
\[
	a^{m}b_{m}(a^{-1}, \lambda)
	\equiv - \Tr(\theta^{\tau_{p}(a)})
		\Tr\left( \frac{(a\lambda+\theta)^{m}}{\theta^{1+\tau_{p}(a)}} \right) \bmod p.
\]
\end{theorem}

\begin{remark} \label{r = 0}
Before the proof of \cref{trace_formula},
recall that
\[
	b_{m}(x, \lambda) = \sum_{j \geq 0} \binom{m}{j} b_{j}(x) \lambda^{m-j}
\]
(cf.\  \cref{convolution}).
Therefore,
it is sufficient to prove \cref{trace_formula} for $\lambda = 0$.
In this case,
the statement itself has already appeared in Junod's paper \cite[Theorem 5]{Junod_trace}.
However, his proof is imprecise.
Moreover, he stated
an intermediate result \cite[Theorem 3]{Junod_trace},
which is incorrect.
In this section, we correct these errors and give self-contained proofs.
The ideas of our proofs are the same as \cite{Junod_trace}, but we shall clarify the crucial ideas.
\end{remark}

\begin{remark} [$p = 2$] \label{p = 2}
The proof of \cref{trace_formula} for $p = 2$ is quite easy:
In this case, we may assume that $a = 1$.
Since $\tau_{2}(1) = 1$,
$b_{n+2} \equiv b_{n+1}+b_{n} \bmod 2$,
and $\theta^{3} = 1$,
it is sufficient to prove the formula for $m = 0, 1, 2$,
which boils down to checking that
\[
	b_{0}(1) = 1, \quad
	b_{1}(1) = 1, \quad
	b_{2}(1) = 2,
\]
and
\[
	\mathrm{Tr(\theta^{-2})}
		\equiv 1  \bmod 2, \quad
	\mathrm{Tr(\theta^{-1})} \equiv 1 \bmod 2, \quad
	\mathrm{Tr(\theta^{0})} \equiv 2 \bmod 2.
\]
\end{remark}

In what follows, we fix an odd prime $p$.
Let $n \in \Z_{\geq 1}$.
Define two polynomials in $\Z[x, z]$ by
\[
	g_{x, n}(z) := (1-z^{p-1})^{n} - (xz)^{np}
	\quad \text{and} \quad
	g^{*}_{x, n}(z) := z^{np}g_{x, n}(z^{-1}) = (z^{p}-z)^{n} - x^{np}.
\]
In what follows, $\overline{F}$ denotes a fixed algebraic closure of each field $F$.
Let $\mathcal{O}$ be the ring of integers in $\overline{\Q_{p}}$
and $\zeta_{n} \in \mathcal{O}^{\times}$ be a primitive $n$-th root of unity.
Note that $g^{*}_{x, n}(z) \in \mathcal{O}[[z]]$ if and only if $x \in \mathcal{O}^{\times}$.


\begin{lemma} \label{separable}
\footnote{
	In the proof of \cite[Theorem 3]{Junod_trace},
	Junod used \cite[Lemma]{Junod_trace},
	which states that
	if one takes $\theta \in \overline{\Q_{p}}$ so that $\theta^{p} = \theta+1$
	and $a \in \Z$ so that $p \nmid a$,
	then the polynomial $g^{*}_{a, n}(z)$ has distinct $np$ roots in the ring $\Z_{p}[\theta]/np\Z_{p}[\theta]$.
	This statement is incorrect when $p \nmid n$
	because we need to extend the ring
	from $\Z_{p}[\theta]/np\Z_{p}[\theta] \simeq \F_{p^{p}}$ to $\F_{p^{p}}(\zeta_{n})$.
	Moreover, his proof is based on the claim that
	$a\theta+m$ ($0 \leq m \leq np-1$) form the distinct $np$ roots of $g^{*}_{a, n}(z)$ in $\overline{\F_{p}}$,
	which is wrong whenever $n > 1$.
	In order to remedy this error,
	we replace \cite[Lemma]{Junod_trace} by \cref{separable}.
}
For every $x \in \mathcal{O}^{\times}$,
$g^{*}_{x, n}(z)$ has distinct $np$ roots in $\mathcal{O}$.
\footnote{
	In view of the proof,
	it is sufficient to assume that $x^{p(p-1)} \not\equiv p(1-p^{-1})^{p-1} \bmod \langle \zeta_{n} \rangle$ (e.g. $x \in \Q_{p}^{\times}$). 
	}
\end{lemma}
	
\begin{proof}
Since $g^{*}_{x, n}(z) = \prod_{m = 1}^{n} g^{*}_{\zeta_{n}^{m}x, 1}(z)$
and any distinct two factors $g^{*}_{\zeta_{n}^{m}x, 1}(z)$ ($1 \leq m \leq n$) have no common roots,
it is sufficient to prove that each factor $g^{*}_{\zeta_{n}^{m}x, 1}(z)$ has $p$ distinct roots in $\mathcal{O}^{\times}$,
i.e., $g^{*}_{\zeta_{n}^{m}x, 1}(z)$ is separable.
Indeed, 
the derivative of $g^{*}_{\zeta_{n}^{m}x, 1}(z)$
\[
	\frac{\partial}{\partial z}g^{*}_{\zeta_{n}^{m}x, 1}(z) = pz^{p-1}-1 
\]
has no common roots with $g^{*}_{\zeta_{n}^{m}x, 1}(z)$ itself
since $x \in \mathcal{O}^{\times}$.
\end{proof}

The following formula,
which is a corrected and generalized version of \cite[Theorem 3]{Junod_trace},
gives a representation of $b_{m}(x)$ for $x \in \mathcal{O}^{\times}$
as a trace-like sum over the roots of the polynomial $g^{*}_{x, n}(z)$.
It is a key ingredient in the proof of \cref{main_2}
and originate from \cite[Lemme 4]{Barsky-Benzaghou},
where the latter was established decades after the studies \cite{Barsky_Bell,Radoux_Artin_Schreier,Radoux_Touchard,Radoux_survey} in 1970'.

\begin{proposition} \label{intermediate}
Let $m \in \Z_{\geq 0}$ and $n \in \Z_{\geq 1}$.
Then, for every $x \in \mathcal{O}^{\times}$,
\[
	-\frac{x^{p-1}}{n} \sum_{l = 0}^{n-1} \zeta_{n}^{l}
		\sum_{\substack{\Theta \in \overline{\Q_{p}} \\ \Theta^{p}-\Theta = \zeta_{n}^{l}x^{p}}}
		\frac{\Theta^{m}}{1-p\Theta^{p-1}}
		\sum_{k = 0}^{np-1} \prod_{j = 1}^{k} \left( \frac{\Theta}{x}+\frac{j}{x} \right)
	\in \Z_{p}[x]
\]
and it congruent to $b_{m}(x)$ in $\Z_{p}[x]/np\Z_{p}[x]$.
In particular, for every $(p-1)$-th root of unity $\xi \in \Z_{p}^{\times}$,
the following congruence holds in $\Z_{p}/np\Z_{p}$:
\[
	\xi^{m}b_{m}(\xi^{-1})
	\equiv -\frac{1}{n} \sum_{l = 0}^{n-1} \zeta_{n}^{l}
		\sum_{\substack{\theta \in \overline{\Q_{p}} \\ \theta^{p}-\theta = \zeta_{n}^{l}}}
		\frac{\theta^{m}}{1-p\theta^{p-1}} \sum_{k = 0}^{np-1} \prod_{j = 1}^{k} (\theta+j\xi)
		\bmod np\Z_{p}.
\]
\end{proposition}	
	
\begin{proof}
Define a rational function in $\overline{\Q_{p}}(z)$ by
\[
	F_{x, n}(z)
	:= \frac{1}{g_{x, n}(z)} \sum_{k = 0}^{np-1} (xz)^{np-k-1} \prod_{j = 1}^{k} (1+jz).
\]
Then,
\cref{separable} implies that
$F_{x, n}(z)$ has a partial fractional decomposition as follows
\[
	F_{x, n}(z)
	= \sum_{\substack{\eta \in \overline{\Q_{p}} \\ g_{x, n}(\eta) = 0}} \frac{\rho_{x, n}(\eta)}{z-\eta},
\]
where $\rho_{x, n}(\eta) := \left. (z-\eta)F_{x, n}(z) \right|_{z = \eta}$ (cf.\  the proof of \cref{GF_Bell_Mezo}).
In particular, we obtain
\[
	\left. \frac{1}{m!} \frac{\partial^{m}}{\partial z^{m}} F_{x, n}(z) \right|_{z = 0}
	= - \sum_{\substack{\eta \in \overline{\Q_{p}} \\ g_{x, n}(\eta) = 0}} \frac{\rho_{x, n}(\eta)}{\eta^{m+1}}.
\]
Moreover,
for every root $\eta$ of $g_{x, n}(z)$,
we have
\[
	\frac{\partial}{\partial z}g_{x, n}(\eta)
	= -n(p-1)(1-\eta^{p-1})^{n-1}\eta^{p-2} - npx(x\eta)^{np-1}
		= -nx(x\eta)^{np-1}\frac{\eta^{p-1}-p}{\eta^{p-1}-1},
\]
and hence
\[
	\rho_{x, n}(\eta)
	= \frac{1}{\frac{\partial}{\partial z} g_{x, n}(\eta)} \sum_{k = 0}^{np-1} (x\eta)^{np-k-1} \prod_{j = 1}^{k} (1+j\eta)
	= -\frac{\eta^{p-1}-1}{nx(\eta^{p-1}-p)} \sum_{k = 0}^{np-1} \prod_{j = 1}^{k} \left( \frac{1}{x\eta}+\frac{j}{x} \right).
\]
By combining the above calculations and rewriting $\eta^{-1}$ to $\Theta$,
we obtain the following equality:
\begin{align*}
	\left. \frac{1}{m!} \frac{\partial^{m}}{\partial z^{m}} F_{x, n}(z) \right|_{z = 0}
	&= \frac{1}{nx} \sum_{\substack{\eta \in \overline{\Q_{p}} \\ g_{x, n}(\eta) = 0}}
		\frac{\eta^{p-1}-1}{\eta^{m+1}(\eta^{p-1}-p)}
		\sum_{k = 0}^{np-1} \prod_{j = 1}^{k} \left( \frac{1}{x\eta} + \frac{j}{x} \right) \\
	&= \frac{1}{nx} \sum_{\substack{\Theta \in \overline{\Q_{p}} \\ g^{*}_{x, n}(\Theta) = 0}}
		\frac{\Theta^{m}(\Theta-\Theta^{p})}{1-p\Theta^{p-1}}
		\sum_{k = 0}^{np-1}\prod_{j = 1}^{k} \left( \frac{\Theta}{x} + \frac{j}{x} \right) \\
	&= -\frac{x^{p-1}}{n} \sum_{l = 0}^{n-1} \zeta_{n}^{l}
		\sum_{\substack{\Theta \in \overline{\Q_{p}} \\ \Theta^{p}-\Theta = \zeta_{n}^{l}x^{p}}}
		\frac{\Theta^{m}}{1-p\Theta^{p-1}}
		\sum_{k = 0}^{np-1} \prod_{j = 1}^{k} \left( \frac{\Theta}{x}+\frac{j}{x} \right).
\end{align*}
Since $x \in \mathcal{O}^{\times}$,
the leftmost side belongs to $\Z_{p}[x]$,
which implies the first statement.
On the other hand,
\cref{rationality} implies that
\[
	b_{m}(x)
	\equiv \left. \frac{1}{m!} \frac{\partial^{m}}{\partial z^{m}} F_{x, n}(z) \right|_{z = 0} \bmod np\Z_{p}[x].
\]
Therefore, we obtain the second statement.
\footnote{
	In this step, Junod made a crucial error
	to substitute the congruence
	\[
		\frac{\partial}{\partial z}g_{x, n}(z)
			= -nx(x\eta)^{np-1}\frac{\eta^{p-1}-p}{\eta^{p-1}-1}
				\equiv -nx(x\eta)^{np-1}\frac{\eta^{p-1}}{\eta^{p-1}-1} \bmod np\Z[x]
	\]
	to $\rho_{x, n}(\eta) = (\frac{\partial}{\partial z} g_{x, n}(\eta))^{-1} \sum_{k = 0}^{np-1} (x\eta)^{np-k-1} \prod_{j = 1}^{k} (1+j\eta)$.
	Since $(\partial/\partial z)g_{x, n}(z) \in \mathcal{O}^{\times}$ only if $p \nmid n$,
	the above substitution cannot be justified.
}
For the last statement,
it is sufficient to note that,
$\theta^{p}-\theta = \zeta_{n}^{l}$ if and only if
$(\xi^{-1}\theta)^{p}-(\xi^{-1}\theta) = \zeta_{n}^{l}(\xi^{-1})^{p}$.
\end{proof}

\begin{remark} \label{transcendence_F}
The above proof relies on \cref{rationality},
which lifts the ordinary generating series $F(x, 0, z) \in (\Z[x]/np\Z[x])[[z]]$ of $(b_{m}(x) \bmod np\Z[x])_{m}$
to a rational function $F_{x, n}(z) \in \Q_{p}(x, z)$.
Such a rational approximation goes back to Barsky's pioneering work \cite{Barsky_Bell}.
Note that $F(x, 0, z)$ itself is highly transcendental
because even a specialization $F(1, 0, z) = e^{-1}(1-z)^{-1}{}_{1}F_{1}(\begin{smallmatrix} -z^{-1} \\ 1-z^{-1} \end{smallmatrix} \mid 1)$ 
satisfies no algebraic differential equations over $\C[[z]]$ \cite[Theorem 3.5]{Klazar}.
\end{remark}

In the proof of \cref{trace_formula},
we shall use the following lemma,
which is a part of \cite[Proposition 4]{Junod_trace}.
The author believes that the following proof is more natural than \cite{Junod_trace} involving more technical calculations.
Additionally, the author expects that
our proof would lead us to deeper understanding of \cref{trace_formula}
because it relates $\tau_{p}(a) := \sum_{j = 1}^{p-1} jp^{ja-1}$ more directly to
the cyclotomic polynomial $(x^{p}-1)/(x-1)$ than \cite{Junod_trace}.

\begin{lemma} \label{inverse}
Suppose that $a \not\equiv 0 \bmod p$.
Then, it holds that
\[
	\tau_{p}(a)(p^{a}-1) \equiv 1 \bmod \frac{p^{p}-1}{p-1}.
\]
In particular, if $\eta \in \F_{p^{p}}$ has norm $1$,
then $\eta^{\tau_{p}(a)}$ is the unique $(p^{a}-1)$-th root of $\eta$ in $\F_{p^{p}}$.
\end{lemma}

\begin{proof}
By calculating the derivative of $x^{ap}-1 \in \Z[x]$ in two ways,
we obtain
\[
	 \left( \frac{x^{ap}-1}{x^{a}-1} \right)' (x^{a}-1)
		+ \frac{x^{ap}-1}{x^{a}-1} \cdot ax^{a-1}
			= apx^{ap-1}.
\]
Since $a \not\equiv 0 \bmod p$, 
there exists a polynomial $\psi(x) \in \Z[x]$ such that
$(x^{ap}-1)/(x^{a}-1) = \psi(x)(x^{p}-1)/(x-1)$,
and hence
\[
	 \left. \left( \frac{x^{ap}-1}{x^{a}-1} \right)'\right|_{x = p} (p^{a}-1)
		+ ap^{a-1}\psi(p) \cdot \frac{p^{p}-1}{p-1}
			= a \cdot (p^{p})^{a}.
\]
By taking modulo $(p^{p}-1)/(p-1)$, we obtain the desired congruence.
\end{proof}

\begin{proof} [Proof of \cref{trace_formula}]
As noted in \cref{r = 0,p = 2},
we may assume that $\lambda = 0$ and $p \geq 3$.
By applying \cref{intermediate} for $n = 1$
and \cref{inverse},
we obtain a congruence in $\Z_{p}/p\Z_{p} \simeq \F_{p}$
\begin{align*}
	a^{m}b_{m}(a^{-1})
	&\equiv -\sum_{\substack{\theta \in \F_{p^{p}} \\ g^{*}_{1, 1}(\theta) = 0}}
		\theta^{m} \sum_{k = 0}^{p-1} \prod_{j = 1}^{k} (\theta+ja)
	\equiv -\sum_{\substack{\theta \in \F_{p^{p}} \\ \theta^{p} = \theta+1}}
		\theta^{m-1} \sum_{k = 0}^{p-1} \prod_{j = 0}^{k} \theta^{p^{ja}} \\
	&\equiv -\sum_{\substack{\theta \in \F_{p^{p}} \\ \theta^{p} = \theta+1}}
		\theta^{m-1} \sum_{k = 0}^{p-1} \theta^{\frac{p^{a(k+1)}-1}{p^{a}-1}}
	\equiv -\sum_{\substack{\theta \in \F_{p^{p}} \\ \theta^{p} = \theta+1}}
		\theta^{m-1-\tau_{p}(a)} \sum_{k = 0}^{p-1} (\theta^{p^{a(k+1)}})^{\tau_{p}(a)} \bmod p.
\end{align*}
Since $a \not\equiv 0 \bmod p$,
$\theta^{p^{a(k+1)}}$ ($0 \leq k \leq p-1$) form the roots of the polynomial $g^{*}_{1,1}(z) = z^{p}-z-1$,
and hence $\sum_{k = 0}^{p-1} (\theta^{p^{a(k+1)}})^{\tau_{p}(a)} = \Tr(\theta^{\tau_{p}(a)}) \in \F_{p}$.
This completes the proof.
\end{proof}

\section{Proof of \cref{main_1}}

In this section, we complete the proof of \cref{main_1}.
Let $p$ be an odd prime.
Then, since $\tau_{p}(a) \equiv (p-1)/2 \bmod p-1$ and $\Tr(\theta^{-1}) = -1$,
\cref{trace_formula} for $m = \tau_{p}(a)$ implies that
\begin{equation} \label{Bell_trace}
	\left( \frac{a}{p} \right)_{2} \Tr(\theta^{\tau_{p}(a)})
	\equiv b_{\tau_{p}(a)}(a^{-1}) \bmod p,
\end{equation}
where $(a/p)_{2}$ denotes the quadratic residue symbol modulo $p$.
On the other hand,
by \cite{Mordell_factorial,Chowla},
we have
\[
	\left( \frac{p-1}{2} \right)!
	\equiv \begin{cases}
	(-1)^{\frac{h(p)+1}{2}}t_{p} \bmod p & \text{if $p \equiv 1 \bmod 4$}, \\
	(-1)^{\frac{h(-p)+1}{2}} \bmod p & \text{if $p \equiv -1 \bmod 4$ and $p > 3$}.
	\end{cases}
\]
Therefore, \cref{main_1} follows from the following congruence.

\begin{theorem} \label{Bell_factorial}
Let $p$ be an odd prime
and $a \in \Z_{\geq 1}$ such that $p \nmid a$.
Then, it holds that
\[
	b_{\tau_{p}(a)}(a^{-1})
	\equiv \left( \frac{-2}{p} \right)_{2} \left( \frac{p-1}{2} \right)! \bmod p.
\]
\end{theorem}

For $a = 1$, the above congruence was already obtained by Radoux
in \cite[p.\ 881]{Radoux_Artin_Schreier} and \cite[(13)]{Radoux_Hankel},
whose proof based on the following Hankel determinant formula.

\begin{lemma} [{\cite[(12)]{Radoux_Hankel}, \cite[Theorem 1]{Ehrenborg}, see also \cite{Aigner}}] \label{Sylvester_type}
Let $n \in \Z_{\geq 1}$,
and set $n \times n$ matrix $\mathbb{B}_{n}(x) := (b_{i+j}(x))_{0 \leq i, j \leq n-1}$.
Then, we have
\[
	\det\mathbb{B}_{n}(x) = \left( \prod_{j = 0}^{n-1} j! \right) x^{\frac{n(n-1)}{2}}.
\]
In particular,
if $n = p$ is an odd prime and $a \in \Z$ not divisible by $p$,
then it holds that
\[
	\det\mathbb{B}_{p}(a^{-1})
	\equiv \left( \frac{2a}{p} \right)_{2} \left( \frac{p-1}{2} \right)! \bmod p.
\]
\end{lemma}

We also use the following lemma,
whose proof is given in the appendix.

\begin{lemma} [{cf.\ \cite[Lemme 3]{Barsky-Benzaghou}}] \label{trace_calculus}
Let $p$ be an odd prime,
and $n \in \Z$.
Then, $\Tr(\theta^{n})$ coincides with $-1$ times the coefficient of $\theta^{p-1}$
in the $\F_{p}$-linear representation of $\theta^{n}$
for the basis consisting of $\theta^{0} = 1, \theta, ..., \theta^{p-1}$.
In particular, we have
\[
	\Tr(\theta^{0})
	= \Tr(\theta^{1})
	= \cdots
	= \Tr(\theta^{p-2})
	= \Tr(\theta^{p})
	= \Tr(\theta^{p+1})
	= \cdots
	= \Tr(\theta^{2p-3}) = 0
\]
and
\[
	\Tr(\theta^{-1})
	= \Tr(\theta^{p-1}) = -1. 	
\]
\end{lemma}

Now, we can prove \cref{Bell_factorial} following Radoux's idea.

\begin{proof} [Proof of \cref{Bell_factorial}]
Set $\mathbb{B}_{p}^{(n)}(x) := (b_{n+i+j}(x))_{0 \leq i, j \leq p-1}$.
Then, \cref{Touchard_congruence} shows that
\begin{align*}
	\det\mathbb{B}_{p}^{(n+1)}(x)
	&= \det \left( \begin{matrix}
	b_{n+1}(x) & b_{n+2}(x) & \cdots & b_{n+p}(x)-b_{n+1}(x) \\
	b_{n+2}(x) & b_{n+3}(x) & \cdots & b_{n+p+1}(x)-b_{n+2}(x) \\
	\vdots & \vdots & & \vdots \\
	b_{n+p}(x) & b_{n+p+1}(x) & \cdots & b_{n+2p-1}(x)-b_{n+p}(x) \\
	\end{matrix} \right) \\
	%
	%
	&\equiv x^{p} \cdot (-1)^{p-1}\det\mathbb{B}_{p}^{(n)}(x)
	\equiv \cdots
	\equiv x^{pn}\det\mathbb{B}_{p}^{(0)}(x) \bmod p.
\end{align*}
Since $\tau_{p}(a) \equiv (p-1)/2 \bmod p-1$
and $a^{\tau_{p}(a)} \equiv (a/p)_{2} \bmod p$,
\cref{Sylvester_type} implies that
\[
	\det\mathbb{B}_{p}^{(\tau_{p}(a))}(a^{-1})
	\equiv a^{-p\tau_{p}(a)} \det\mathbb{B}_{p}^{(0)}(a^{-1})
	\equiv \left( \frac{2}{p} \right)_{2} \left( \frac{p-1}{2} \right)! \bmod p.
\]
On the other hand,
by \cref{trace_formula} and \cref{trace_calculus},
we have
\[
	b_{\tau_{p}(a)+1}(a^{-1})
	\equiv \cdots
	\equiv b_{\tau_{p}(a)+p-1}(a^{-1})
	\equiv b_{\tau_{p}(a)+p+1}(a^{-1})
	\equiv \cdots
	\equiv b_{\tau_{p}(a)+2p-1}(a^{-1})
	\equiv 0 \bmod p
\]
(cf.\ \cite{Kahale,Radoux_Artin_Schreier} for $a = 1$),
and hence
\begin{align*}
	\det\mathbb{B}_{p}^{(\tau_{p}(a))} (a^{-1})
	&\equiv \det\left( \begin{matrix}
	b_{\tau_{p}(a)}(a^{-1}) & 0 & \cdots & 0 & 0 \\
	0 & 0 & \cdots & 0 & a^{-p}b_{\tau_{p}(a)}(a^{-1})  \\
	0 & 0 & \cdots & a^{-p}b_{\tau_{p}(a)}(a^{-1})  & 0 \\
	\vdots & \vdots &  & \vdots & \vdots \\
	0 & a^{-p}b_{\tau_{p}(a)}(a^{-1})  & \cdots & 0 & 0 \\
	\end{matrix} \right) \\
	&\equiv (-1)^{\frac{p-1}{2}}b_{\tau_{p}(a)}(a^{-1}) \bmod p.
\end{align*}
By comparing the above two formulae,
we obtain the desired congruence.
\end{proof}

\begin{remark}
In \cite[Theorem 5]{Junod_trace},
Junod stated his trace formula in the form
\[
	a^{m}b_{m}(a^{-1})
	\equiv -a^{\tau_{p}(1)} b_{\tau_{p}(1)} \Tr(\theta^{m-1-\tau_{p}(a)})
		\bmod p.
\]
It follows from \cref{trace_formula} and the fact that
$(a/p)_{2}\Tr(\theta^{\tau_{p}(a)})$ is independent of $a$ (but has deep arithmetic nature as stated in \cref{main_1}).
However, Junod's proof of the latter fact
depends on the assumption that $(a^{-1}\theta)^{p} \equiv a^{-1}\theta+1 \bmod p$,
which is wrong unless $a \equiv 1 \bmod p$.
Now, we can fill the gap 
by combining \cref{Bell_trace} and \cref{Bell_factorial}.
\end{remark}

\begin{remark}
The right hand side of \cref{Bell_factorial},
or equivalently $b_{\tau_{p}(a)}(a^{-1}) \bmod p$,
can be described by the special value $\Gamma_{p}(1/2)$ of Morita's $p$-adic $\Gamma$-function \cite{Morita},
and hence by a quadratic Gauss sum \cite[Theorem 1.7]{Gross-Koblitz}.
In fact, these quantities are
strongly related to the Artin-Schreier curve defined by $x^{p}-x = y^{2}$ \cite[\S\S 15--16]{Lang}.
Hence, it should be an interesting problem to give geometric interpretations of more general $b_{m}(a^{-1}) \bmod p$.
\end{remark}

\appendix
\section{Proof of \cref{trace_calculus}}

Set $a_{n, 0}, ..., a_{n, p-1} \in \F_{p}$ so that
$\theta^{n} = \sum_{i = 0}^{p-1} a_{n, i} \theta^{i}$.
Then,
for every $m \in \Z_{\geq 0}$,
we have
\[
	\theta^{np^{m}}
	= \left( \sum_{j = 0}^{p-1} a_{n, j} \theta^{j} \right)^{p^{m}}
		= \sum_{j = 0}^{p-1} a_{n, j} (\theta+m)^{j}
			= \sum_{i = 0}^{p-1} \left( \sum_{j = i}^{p-1} \binom{j}{i} a_{n, j} m^{j-i} \right) \theta^{i}.
\]
On the other hand,
since $\theta^{0} = 1, \theta, ..., \theta^{p-1}$ are linearly independent over $\F_{p}$,
we see that
\[
	\Tr(\theta^{n})
	= \sum_{m = 0}^{p-1} \theta^{np^{m}}
		= \sum_{m = 0}^{p-1} \sum_{j = 0}^{p-1} \binom{j}{0} a_{n, j} m^{j-0}
			= \sum_{j = 0}^{p-1} a_{n, j} \sum_{m = 0}^{p-1} m^{j}.
\]
Therefore,
the claimed identity follows from the following congruence:
\[
	S_{j}(p)
	:= \sum_{m = 0}^{p-1} m^{j}
	\equiv \begin{cases}
	-1 \bmod p & \text{if $j > 0$ and $j \equiv 0 \bmod p-1$}  \\
	0 \bmod p & \text{otherwise}
	\end{cases}.
\]
Since the case of $j \equiv 0 \bmod p-1$ is obvious,
we may assume that $1 \leq j \leq p-2$.
By summing up
\[
	\sum_{m = 0}^{j} \binom{j+1}{m} k^{m}
	= (k+1)^{j+1} - k^{j+1}
		\quad (k \in \Z_{\geq 0})
\]
with respect to $0 \leq k \leq p-1$,
we have
\[
	\sum_{m = 0}^{j} \binom{j+1}{m} S_{m}(p)
	= p^{j+1}
	\equiv 0 \bmod p.
\]
By induction on $j$, we obtain the desired congruence.

\section*{Acknowledgements}
The author would like to thank
Ken-ichi Bannai and anonymous referees for their careful reading of early manuscripts of this paper
and giving many valuable comments.
The author would like to thank also
Takashi Hara, Tomoya Machide, Hideki Matsumura, and Takuya Aoki
for their interest in this work.
The author would like to thank also Masataka Ono
for pointing out errors in an early manuscript of this paper.

\end{document}